\documentclass[11pt]{amsart}

\usepackage{amsmath}
\usepackage{amsthm}
\usepackage{amssymb}
\usepackage{enumitem}
\usepackage{lipsum}
\usepackage{mathtools}
\usepackage{tikz}

\tikzstyle{BlueLine}=[thick,color=blue,text=black]
\tikzstyle{BluePoly}=[BlueLine,fill=blue!20]
\tikzstyle{RedLine}=[thick,color=red,text=black]
\tikzstyle{RedPoly}=[RedLine,fill=red!20]
\tikzstyle{GreenLine}=[thick,color=black!30!green,text=black]
\tikzstyle{OrangeLine}=[thick,color=orange]
\tikzstyle{GrayLine}=[thick,color=gray]
\tikzstyle{GrayPoly}=[GrayLine,fill=gray!20]
\tikzstyle{dot}=[shape=circle,draw,color=black,fill=black,inner sep=1pt]
\tikzstyle{littledot}=[dot,inner sep=0.75pt]
\tikzstyle{disk}=[thick,shape=circle,draw,color=black,fill=yellow!20]
\tikzstyle{plate}=[thick,shape=rectangle,draw,color=black,fill=yellow!20,rounded corners,minimum size=1.1cm]

\newcommand{\makepoints}{ 
  \foreach \n in {1,...,4} {
    \coordinate (\n) at (\n*-90-135:0.5cm);
  }; 
}
\newcommand{\drawpoints}{ 
  \foreach \n in {1,...,4} {
    \draw (\n) node [littledot] {};
  };
}

\newcommand{\makepent}{ 
  \foreach \n in {1,...,5} {
    \coordinate (\n) at (\n*-72+162:0.5cm);
  };
}
\newcommand{\drawpent}{ 
  \foreach \n in {1,...,5} {
    \draw (\n) node [littledot] {};
  };
}

\theoremstyle{plain}
\newtheorem{thm}{Theorem}[section]

\newtheorem{prop}[thm]{Proposition}
\newtheorem*{conj}{Conjecture}

\newtheorem{thma}{Theorem}

\theoremstyle{definition}
\newtheorem{defin}[thm]{Definition}

\theoremstyle{remark}
\newtheorem*{rem}{Remark}

\newcommand{\NC}{\textsc{NC}}
\newcommand{\PF}{\textsc{PF}}
\newcommand{\RF}{\textsc{RF}}
\newcommand{\pos}{\textsc{Poset}}
\newcommand{\sym}{\textsc{Sym}}

\begin{document}

\title{A Decomposition of Parking Functions by Undesired Spaces}
\author[M. Bruce]{Melody Bruce}
\author[M. Dougherty]{Michael Dougherty}
\author[M. Hlavacek]{Max Hlavacek}
\author[R. Kudo]{Ryo Kudo}
\author[I. Nicolas]{Ian Nicolas}
\date{\today}

\begin{abstract}
  There is a well-known bijection between parking functions of a fixed length and maximal chains of the noncrossing partition lattice which we can use to associate to each set of parking functions a poset whose Hasse diagram is the union of the corresponding maximal chains. We introduce a decomposition of parking functions based on the largest number omitted and prove several theorems about the corresponding posets. In particular, they share properties with the noncrossing partition lattice such as local self-duality, a nice characterization of intervals, a readily computable M\"obius function, and a symmetric chain decomposition. We also explore connections with order complexes, labeled Dyck paths, and rooted forests.
\end{abstract}

\maketitle

\section*{Introduction}

An $n$-tuple of integers is called a \emph{parking function} if it can be rearranged so that its $i$th entry is at most $i$. We can then record which of the missing numbers (if any) is the largest, which gives a natural yet little-studied decomposition of parking functions into $n$ parts. These sets can be constructed recursively by noticing that any parking function for which $k<n$ is the largest missing number can be obtained from a parking function of length $n-1$ by adding an $n$, so we are particularly interested in the set of all parking functions of length $k$ which omit $k$. 

By utilizing the well-known correspondence between parking functions and noncrossing partitions \cite{stanley97}, we can view elements from our decomposition as maximal chains in the noncrossing partition lattice $\NC_{n+1}$. Focusing on chains which come from parking functions where $k$ is the largest missing number yields a poset $\pos(\PF_{n,k})$ with elements from $\NC_{n+1}$ (although this is not an induced subposet). From this viewpoint, we can describe the observation above by noting that  $\pos(\PF_{n,k})$ is the direct product of $\pos(\PF_{k,k})$ with a Boolean lattice of rank $n-k$. Further
investigation of these posets yields a surprising number of nice results. 

Following the example set by $\NC_{n+1}$, we present several theorems which underscore the appeal of these posets, the first of which depicts the stucture of their intervals.

\begin{thma}
  Each interval in $\pos(\PF_{n,k})$ is either of the form 
  \[\pos(\PF_{r,r}) \times \prod_{i=1}^j \NC_{m_i} \text{ or } \prod_{i=1}^j \NC_{m_i}.\]
\end{thma}

Similar to the story for $\NC_{n+1}$, we can leverage this result to see the rich symmetry exhibited by these posets.

\begin{thma}
  $\pos(\PF_{n,k})$ is locally self-dual.
\end{thma}

This result is pleasing in its own right, but we can also use it to prove several other facts about the structure of these posets. In particular, we find that each $\pos(\PF_{n,n})$ is irreducible under direct product and use this to prove that there are no non-trivial order-preserving automorphisms of $\pos(\PF_{n,n})$. Perhaps most interestingly, we use our findings to analyze the M\"obius function of our posets.

\begin{thma}
  The M\"obius function of $\pos(\PF_{n,k})$ is zero.
\end{thma}

If we consider the topology of these posets, this computation becomes circumstantial evidence for the homotopy type of the corresponding simplicial complex.

\begin{conj}
  The order complex of $\pos(\PF_{n,k})$ without its bounding elements is contractible.
\end{conj}

In addition to these results, we can follow the decomposition of parking functions through to other common Catalan-type objects such as labeled Dyck paths and labeled rooted forests. In each of these settings, we illustrate the sets which correspond to our decomposition and find a natural way to view the recursion mentioned above, lending weight to the notion that this method illustrates some interesting information.

In the first section of this article, we review the basic definitions for parking functions and introduce the proposed decomposition. Section 2 is dedicated to the posets associated to our decomposition and our proofs of their properties. We discuss the implications regarding order complexes that our results appear to make in Section 3, and we illustrate our decomposition in the areas of labeled Dyck paths and labeled rooted forests in Section 4.

\setcounter{thma}{0}

\section{Background}

We begin by reviewing the basic definitions and introducing our proposed decomposition.

\begin{defin}
  An $n$-tuple of integers $(a_1,\ldots ,a_n)$ is a \emph{parking function} if there is a permutation $\sigma \in Sym_n$ such that $a_{\sigma(1)} \leq \ldots \leq a_{\sigma(n)}$ and $a_{\sigma(i)} \leq i$ for all $i\in [n]$, where $[n]=\{1,\ldots,n\}$. A parking function is called \emph{primitive} if it is already in weakly increasing order - these form the canonical representatives of the equivalence classes formed under the relation of permutation. A clever argument \cite{enum2} shows that there are $(n+1)^{n-1}$ parking functions of length $n$ and $C_n = \frac{1}{n+1}\binom{2n}{n}$ primitive parking functions, where $C_n$ is the $n$th \emph{Catalan number}. We denote the set of all parking functions of length $n$ by $\PF_n$. 
\end{defin}

With these basic properties out of the way, we can define our decomposition of $\PF_n$ into $n$ subsets.

\begin{defin}
\label{parking}
  Let $n,k\in \mathbb{N}$ with $1 < k \leq n$ and define $\PF_{n,k}$ to be the set of all parking functions of length $n$ for which $k$ is the largest missing number. That is,
  \[\PF_{n,k} = \{a \in \PF_n \mid k\not\in a \text{ but } k+1,\ldots , n \in a\}.\]
We can then identify the remaining parking functions with permutations of $[n]$ by thinking of each 
$(a_1,\ldots ,a_n)$ as the element in $\sym_n$ which sends $i$ to $a_i$. Then $\PF_n$ is the disjoint 
union
\[\PF_n = \sym_n \sqcup \PF_{n,2} \sqcup \ldots \sqcup \PF_{n,n}.\]
For example, we can decompose $\PF_3$ as follows:
\begin{align*}
\sym_3 &= \{(1,2,3),(1,3,2),(2,1,3),(2,3,1),(3,1,2),(3,2,1)\} \\
\PF_{3,2} &= \{(1,1,3),(1,3,1),(3,1,1)\} \\
\PF_{3,3} &= \{(1,1,1), (1,1,2), (1,2,1), (2,1,1), (1,2,2), (2,1,2), (2,2,1)\}
\end{align*}
\end{defin}

\begin{rem}
  There is a recursive structure to these subsets: given an element of $\PF_{n,k}$, we can insert $``n+1"$ to create $n+1$ different elements of $\PF_{n+1,k}$. In fact, this is the only way to create elements in $\PF_{n+1,k}$, which reduces some of our work to the case when $k=n$, where we have 
  \[\PF_{n,n} = \{a \in \PF_n \mid n \not\in a\}.\]
  In many of our proofs for properties of $\PF_{n,k}$, it will suffice to demonstrate them for $\PF_{n,n}$.
\end{rem}

\begin{prop}
  The number of parking functions of length $n$ for which $k$ is the largest missing number is
  \[|\PF_{n,k}| = \frac{n!}{k!}((k+1)^{k-1} - k^{k-1}).\]
\end{prop}
\begin{proof}
  By the remark above, we know that $|\PF_{n,k}| = n|\PF_{n-1,k}|$. Then \[|\PF_{n,k}| = \frac{n!}{k!} |\PF_{k,k}|\] so it suffices to compute $|\PF_{k,k}|$. We know that $|\PF_k| = (k+1)^{k-1}$, and since every parking function with a $k$ can be obtained uniquely from $\PF_{k-1}$ by inserting a $k$, there are $k|\PF_{k-1}| = k\cdot k^{k-2}$ parking functions in $\PF_k$ with a $k$. Hence,
  \[|\PF_{k,k}| = (k+1)^{k-1} - k^{k-1}\]
  and the claim follows.
\end{proof}

\section{Connections with $\NC_n$}

In addition to deserving study in their own right, parking functions have an intimate relationship with the lattice of noncrossing partitions. This setting proves an interesting one to consider our decomposition of $\PF_n$.

\begin{figure}
  \begin{tikzpicture}
    \begin{scope}[shift={(-2,0)},BluePoly]
      \node () [disk,minimum size=3.2cm] {};
      \foreach \n in {1,...,6} {
        \coordinate (\n) at (\n*-60-120:1cm);
        \node[black] () at (\n*-60-120:1.3cm) {\n};
      }
      \filldraw (1) -- (2) -- (3) -- cycle;
      \draw (4) -- (6);
      \foreach \n in {1,...,6} {
        \draw (\n) node [dot] {};
      }
    \end{scope}

    \begin{scope}[shift={(2,0)},BluePoly]
      \node () [disk,minimum size=3.2cm] {};
      \foreach \n in {1,...,6} {
        \coordinate (\n) at (\n*-60-120:1cm);
        \node[black] () at (\n*-60-120:1.3cm) {\n};
      }
      \filldraw (5) -- (2) -- (3) -- cycle;
      \draw (4) -- (6);
      \foreach \n in {1,...,6} {
        \draw (\n) node [dot] {};
      }
    \end{scope}
  \end{tikzpicture}
  \caption{A noncrossing partition and a crossing partition}
   \label{noncrossing}
\end{figure}
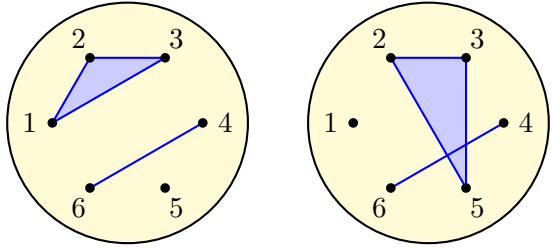

\begin{defin}
  A partition $\sigma$ of $[n]$ is said to be \emph{noncrossing} if, when we consider the elements of 
  $[n]$ arranged in clockwise ascending order around a circle, the convex hulls of the blocks of 
  $\sigma$ are disjoint (Figure \ref{noncrossing}). The poset of all such partitions (a subposet of the
   partition lattice $\Pi_n$) is called the \emph{noncrossing partition lattice} and is denoted $\NC_n$
    (Figure \ref{ncn}). 
  It is well-known that $|\NC_n| = C_n$, the $n$th Catalan number. In addition, the number of maximal chains in $\NC_{n+1}$ (paths from bottom to top in its Hasse diagram) is $(n+1)^{n-1}$.
\end{defin}

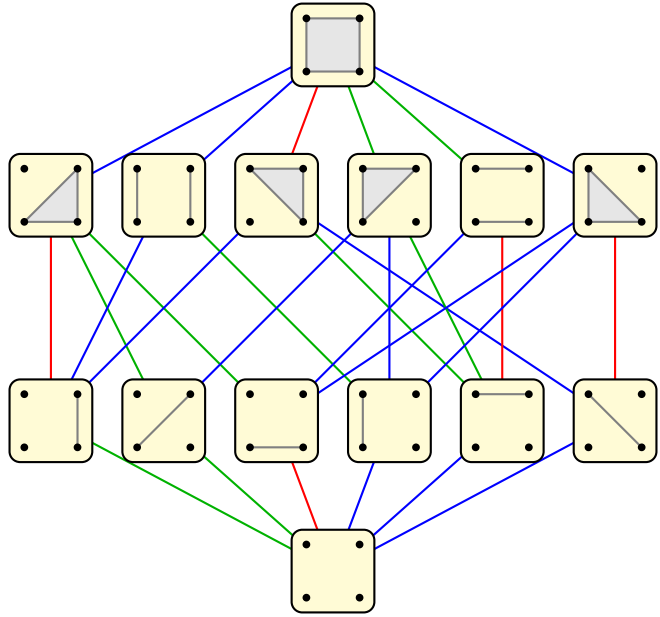
\begin{figure}
\vspace{1em}
\begin{tikzpicture}
	\begin{scope}
		\coordinate (top) at (0,0);
		\foreach \i in {1,...,6} \coordinate (a\i) at (-5.25+\i*1.5,-2);
		\foreach \i in {1,...,6} \coordinate (b\i) at (-5.25+\i*1.5,-5);
		\coordinate (bottom) at (0,-7);
	\end{scope}
	
	\draw[BlueLine] (top) -- (a1);
	\draw[BlueLine] (top) -- (a2);
	\draw[RedLine] (top) -- (a3);
	\draw[GreenLine] (top) -- (a4);
	\draw[GreenLine] (top) -- (a5);
	\draw[BlueLine] (top) --(a6);
		
	\draw [RedLine] (a1) -- (b1);
	\draw [GreenLine] (a1) -- (b2);
	\draw [GreenLine] (a1) -- (b3);
	\draw [BlueLine] (a2) -- (b1);
	\draw [GreenLine] (a2) -- (b4);
	\draw [BlueLine] (a3) -- (b1);
	\draw [GreenLine] (a3) -- (b5);
	\draw [BlueLine] (a3) -- (b6);
	\draw [BlueLine] (a4) -- (b2);
	\draw [BlueLine] (a4) -- (b4);
	\draw [GreenLine] (a4) -- (b5);
	\draw [BlueLine] (a5) -- (b3);
	\draw [RedLine] (a5) -- (b5);
	\draw [BlueLine] (a6) -- (b3);
	\draw [BlueLine] (a6) -- (b4);
	\draw [RedLine] (a6) -- (b6);
		
	\draw [GreenLine] (b1) -- (bottom);
	\draw [GreenLine] (b2) -- (bottom);
	\draw [RedLine] (b3) -- (bottom);
	\draw [BlueLine] (b4) -- (bottom);
	\draw [BlueLine] (b5) -- (bottom);
	\draw [BlueLine] (b6) -- (bottom); 
	
	\begin{scope}
		\node[plate] (top) at (0,0) {};
		\foreach \i in {1,...,6} \node[plate] (a\i) at (-5.25+\i*1.5,-2) {};
		\foreach \i in {1,...,6} \node[plate] (b\i) at (-5.25+\i*1.5,-5) {};
		\node[plate] (bottom) at (0,-7) {};
	\end{scope}
	
	\begin{scope}[GrayPoly]
	\begin{scope}[shift={(0,0)}] \makepoints \filldraw (1) -- (2) -- (3) -- (4) -- cycle; \drawpoints \end{scope}
	
	\begin{scope}[shift={(-3.75,-2)}] \makepoints \filldraw (2) -- (3) -- (4)-- cycle; \drawpoints \end{scope}
	\begin{scope}[shift={(-2.25,-2)}] \makepoints \draw (1) -- (4); \draw (2) -- (3); \drawpoints \end{scope}
	\begin{scope}[shift={(-0.75,-2)}] \makepoints \filldraw (1) -- (2) -- (3) -- cycle; \drawpoints \end{scope}
	\begin{scope}[shift={(0.75,-2)}] \makepoints \filldraw (1) -- (2) -- (4) -- cycle; \drawpoints \end{scope}
	\begin{scope}[shift={(2.25,-2)}] \makepoints \draw (1) -- (2); \draw (3) -- (4); \drawpoints \end{scope}
	\begin{scope}[shift={(3.75,-2)}] \makepoints \filldraw (1) -- (3) -- (4) -- cycle; \drawpoints \end{scope}
	
	\begin{scope}[shift={(-3.75,-5)}] \makepoints \draw (2) -- (3); \drawpoints \end{scope}
	\begin{scope}[shift={(-2.25,-5)}] \makepoints \draw (2) -- (4); \drawpoints \end{scope}
	\begin{scope}[shift={(-0.75,-5)}] \makepoints \draw (3) -- (4); \drawpoints \end{scope}
	\begin{scope}[shift={(0.75,-5)}] \makepoints \draw (1) -- (4); \drawpoints \end{scope}
	\begin{scope}[shift={(2.25,-5)}] \makepoints \draw (1) -- (2); \drawpoints \end{scope}
	\begin{scope}[shift={(3.75,-5)}] \makepoints \draw (1) -- (3); \drawpoints \end{scope}
	
	\begin{scope}[shift={(0,-7)}] \makepoints \drawpoints \end{scope}
	\end{scope}
\end{tikzpicture}
\caption{The noncrossing partition lattice $\NC_4$, where the upper-left vertex is labeled $1$ and proceeds clockwise.}
\label{ncn}
\end{figure}

The reappearance of $(n+1)^{n-1}$ is not a coincidence - the connection which proves this is essential to 
our results.

\begin{thm}
  There is a one-to-one correspondence between parking functions of length $n$ and maximal chains of $\NC_{n+1}$.
\end{thm}

\begin{figure}
  \begin{tikzpicture}
    \begin{scope}[minimum size=1.5cm]
      \node[disk] (top) at (0,0) {};
      \node[disk] (a) at (2.25,0) {};
      \node[disk] (b) at (4.5,0) {};
      \node[disk] (c) at (6.75,0) {};
      \node[disk] (bottom) at (9,0) {};
    \end{scope}
    
    \begin{scope}[GrayPoly]
      \begin{scope}[shift={(9,0)}] \makepent \filldraw (1) -- (2) -- (3) -- (4) -- (5) -- cycle; \drawpent \end{scope}
      \begin{scope}[shift={(6.75,0)}] \makepent \filldraw (1) -- (2) -- (3) -- (5) -- cycle; \drawpent \end{scope}
      \begin{scope}[shift={(4.5,0)}] \makepent \draw (1) -- (2); \draw (3) -- (5); \drawpent \end{scope}
      \begin{scope}[shift={(2.25,0)}] \makepent \draw (3) -- (5); \drawpent \end{scope}
      \begin{scope}[shift={(0,0)}] \makepent \drawpent \end{scope}
    \end{scope}
    
    \begin{scope}[black,thick]
      \draw (top.east) -- (a.west) node [midway,above] {3};
      \draw (a.east) -- (b.west) node [midway,above] {1};
      \draw (b.east) -- (c.west) node [midway,above] {2};
      \draw (c.east) -- (bottom.west) node [midway,above] {3};
    \end{scope}
  \end{tikzpicture}
  \caption{A maximal chain in $\NC_5$ and the corresponding parking function $(3,1,2,3)$, where the top vertex is labeled $1$, proceeding clockwise.}
  \label{edgelabels}
\end{figure}
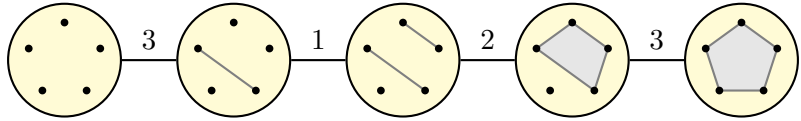

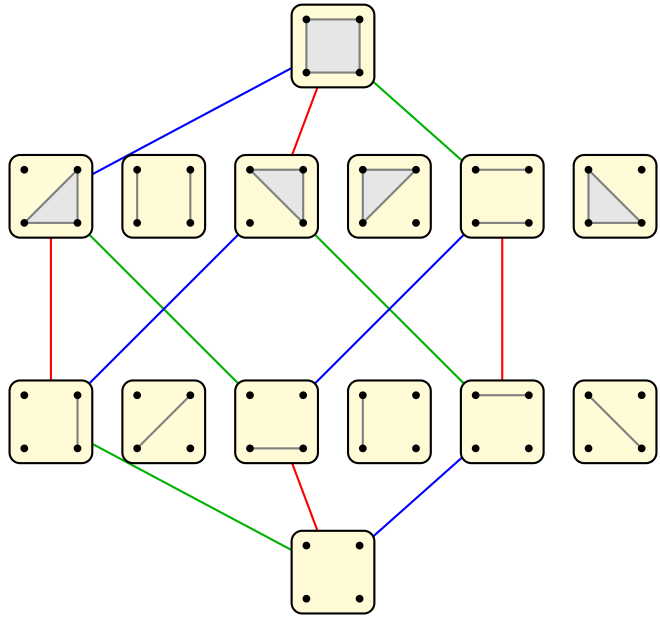
\begin{figure}
\begin{tikzpicture}
	\begin{scope}
		\coordinate (top) at (0,0);
		\foreach \i in {1,...,6} \coordinate (a\i) at (-5.25+\i*1.5,-2);
		\foreach \i in {1,...,6} \coordinate (b\i) at (-5.25+\i*1.5,-5);
		\coordinate (bottom) at (0,-7);
	\end{scope}
	
	\draw [BlueLine] (top) -- (a1);
	\draw [RedLine] (top) -- (a3);
	\draw [GreenLine] (top) -- (a5);
		
	\draw [RedLine] (a1) -- (b1);
	\draw [GreenLine] (a1) -- (b3);
	\draw [BlueLine] (a3) -- (b1);
	\draw [GreenLine] (a3) -- (b5);
	\draw [BlueLine] (a5) -- (b3);
	\draw [RedLine] (a5) -- (b5);
		
	\draw [GreenLine] (b1) -- (bottom);
	\draw [RedLine] (b3) -- (bottom);
	\draw [BlueLine] (b5) -- (bottom);
	
	\begin{scope}
		\node[plate] (top) at (0,0) {};
		\foreach \i in {1,...,6} \node[plate] (a\i) at (-5.25+\i*1.5,-2) {};
		\foreach \i in {1,...,6} \node[plate] (b\i) at (-5.25+\i*1.5,-5) {};
		\node[plate] (bottom) at (0,-7) {};
	\end{scope}
	
	\begin{scope}[GrayPoly]
	\begin{scope}[shift={(0,0)}] \makepoints \filldraw (1) -- (2) -- (3) -- (4) -- cycle; \drawpoints \end{scope}
	
	\begin{scope}[shift={(-3.75,-2)}] \makepoints \filldraw (2) -- (3) -- (4)-- cycle; \drawpoints \end{scope}
	\begin{scope}[shift={(-2.25,-2)}] \makepoints \draw (1) -- (4); \draw (2) -- (3); \drawpoints \end{scope}
	\begin{scope}[shift={(-0.75,-2)}] \makepoints \filldraw (1) -- (2) -- (3) -- cycle; \drawpoints \end{scope}
	\begin{scope}[shift={(0.75,-2)}] \makepoints \filldraw (1) -- (2) -- (4) -- cycle; \drawpoints \end{scope}
	\begin{scope}[shift={(2.25,-2)}] \makepoints \draw (1) -- (2); \draw (3) -- (4); \drawpoints \end{scope}
	\begin{scope}[shift={(3.75,-2)}] \makepoints \filldraw (1) -- (3) -- (4) -- cycle; \drawpoints \end{scope}
	
	\begin{scope}[shift={(-3.75,-5)}] \makepoints \draw (2) -- (3); \drawpoints \end{scope}
	\begin{scope}[shift={(-2.25,-5)}] \makepoints \draw (2) -- (4); \drawpoints \end{scope}
	\begin{scope}[shift={(-0.75,-5)}] \makepoints \draw (3) -- (4); \drawpoints \end{scope}
	\begin{scope}[shift={(0.75,-5)}] \makepoints \draw (1) -- (4); \drawpoints \end{scope}
	\begin{scope}[shift={(2.25,-5)}] \makepoints \draw (1) -- (2); \drawpoints \end{scope}
	\begin{scope}[shift={(3.75,-5)}] \makepoints \draw (1) -- (3); \drawpoints \end{scope}
	
	\begin{scope}[shift={(0,-7)}] \makepoints \drawpoints \end{scope}
	\end{scope}
\end{tikzpicture}
\caption{The maximal chains which form \pos($\sym_3$), i.e those whose labels are a permutation of $(1,2,3)$}
\end{figure}

Given a maximal chain in $\NC_{n+1}$, we produce a parking function by labeling its covering relations, i.e. the edges in the Hasse diagram. For each covering relation $\sigma < \tau$, the larger partition $\tau$ is obtained by joining two blocks $B_1$ and $B_2$ in $\sigma$ to form one block $B$ in $\tau$. Without loss of generality, suppose $B_1$ contains $\min B$; we will then label this edge by the largest number in $B_1$ which is less than all the elements in $B_2$ - see Figure \ref{edgelabels} for an example. Performing this process on each edge and reading the labels on a maximal chain from bottom to top creates a parking function of length $n$, and this map is a bijection. \cite{stanley97}

\begin{defin}
  By the correspondence above, each $\PF_{n,k}$ corresponds to a collection of maximal chains in $\NC_{n+1}$; define $\pos(\PF_{n,k})$ to be the poset whose Hasse diagram is the union of these chains in $\NC_{n+1}$.
\end{defin}

Notice that although $\pos(\PF_{n,k})$ is a subset of $\NC_{n+1}$, it is not an induced subposet in general since there are missing
relations in the smaller poset. In addition, $\pos(\PF_{n,k})$ is not usually a lattice, but there is still a 
nice structure here which deserves exploration.

Since our definition for $\pos(\PF_{n,k})$ is based on the maximum chains of $\NC_{n+1}$, it is not 
immediately clear which noncrossing partitions will appear. Focusing for a moment on 
when $n=k$, we can see that the elements of $\NC_{n+1}$ which do not lie in $\pos(\PF_{n,n})$ are
precisely those for which every maximal chain containing them has an $n$-label on one of its edges. 

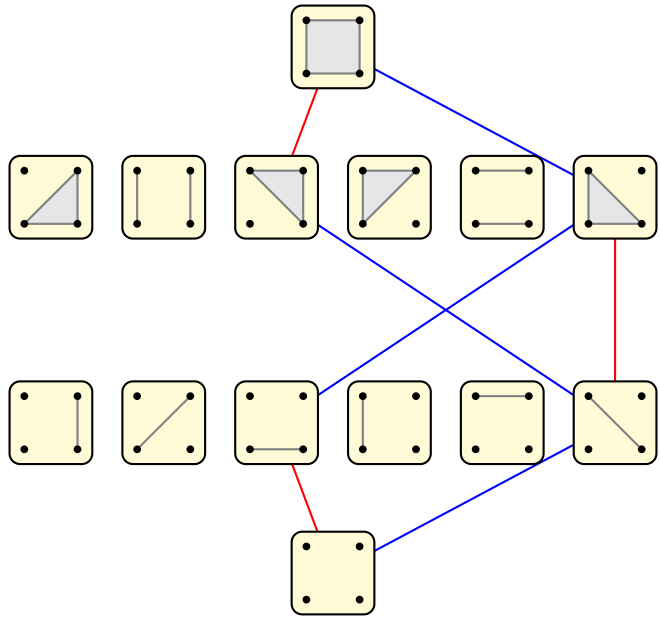
\begin{figure}
\begin{tikzpicture}	
	\begin{scope}
		\coordinate (top) at (0,0);
		\foreach \i in {1,...,6} \coordinate (a\i) at (-5.25+\i*1.5,-2);
		\foreach \i in {1,...,6} \coordinate (b\i) at (-5.25+\i*1.5,-5);
		\coordinate (bottom) at (0,-7);
	\end{scope}
	
	\draw [RedLine] (top.south) -- (a3.north);
	\draw [BlueLine] (top.south) --(a6.north);
		
	\draw [BlueLine] (a3.south) -- (b6.north);
	\draw [BlueLine] (a6.south) -- (b3.north);
	\draw [RedLine] (a6.south) -- (b6.north);
		
	\draw [RedLine] (b3.south) -- (bottom.north);
	\draw [BlueLine] (b6.south) -- (bottom.north); 
	
	\begin{scope}
		\node[plate] (top) at (0,0) {};
		\foreach \i in {1,...,6} \node[plate] (a\i) at (-5.25+\i*1.5,-2) {};
		\foreach \i in {1,...,6} \node[plate] (b\i) at (-5.25+\i*1.5,-5) {};
		\node[plate] (bottom) at (0,-7) {};
	\end{scope}	
	
	\begin{scope}[GrayPoly]
	\begin{scope}[shift={(0,0)}] \makepoints \filldraw (1) -- (2) -- (3) -- (4) -- cycle; \drawpoints \end{scope}
	
	\begin{scope}[shift={(-3.75,-2)}] \makepoints \filldraw (2) -- (3) -- (4)-- cycle; \drawpoints \end{scope}
	\begin{scope}[shift={(-2.25,-2)}] \makepoints \draw (1) -- (4); \draw (2) -- (3); \drawpoints \end{scope}
	\begin{scope}[shift={(-0.75,-2)}] \makepoints \filldraw (1) -- (2) -- (3) -- cycle; \drawpoints \end{scope}
	\begin{scope}[shift={(0.75,-2)}] \makepoints \filldraw (1) -- (2) -- (4) -- cycle; \drawpoints \end{scope}
	\begin{scope}[shift={(2.25,-2)}] \makepoints \draw (1) -- (2); \draw (3) -- (4); \drawpoints \end{scope}
	\begin{scope}[shift={(3.75,-2)}] \makepoints \filldraw (1) -- (3) -- (4) -- cycle; \drawpoints \end{scope}
	
	\begin{scope}[shift={(-3.75,-5)}] \makepoints \draw (2) -- (3); \drawpoints \end{scope}
	\begin{scope}[shift={(-2.25,-5)}] \makepoints \draw (2) -- (4); \drawpoints \end{scope}
	\begin{scope}[shift={(-0.75,-5)}] \makepoints \draw (3) -- (4); \drawpoints \end{scope}
	\begin{scope}[shift={(0.75,-5)}] \makepoints \draw (1) -- (4); \drawpoints \end{scope}
	\begin{scope}[shift={(2.25,-5)}] \makepoints \draw (1) -- (2); \drawpoints \end{scope}
	\begin{scope}[shift={(3.75,-5)}] \makepoints \draw (1) -- (3); \drawpoints \end{scope}
	
	\begin{scope}[shift={(0,-7)}] \makepoints \drawpoints \end{scope}
	\end{scope}
\end{tikzpicture}
\caption{The maximal chains which form $\pos(\PF_{3,2})$}
\end{figure}

A covering relation $\sigma < \tau$ in $\NC_{n+1}$ is labeled $n$ if and only if a block in $\tau$
is the union of the block containing $n$ in $\sigma$ and $\{n+1\} \in \sigma$. There are exactly two ways to guarantee that this will happen in every maximal chain passing through some
$\pi\in \NC_{n+1}$: either $\{n,n+1\}$ is a block in $\pi$ or $1$ and $n$ are in the same block in $\pi$
(written $1\sim n$ in $\pi$ or $1\sim_\pi n$) and $\{n+1\}$ is a block in $\pi$. This can be summarized in a proposition
as follows:

\begin{figure}
\begin{tikzpicture}
	\begin{scope}
		\coordinate (top) at (0,0);
		\foreach \i in {1,...,6} \coordinate (a\i) at (-5.25+\i*1.5,-2);
		\foreach \i in {1,...,6} \coordinate (b\i) at (-5.25+\i*1.5,-5);
		\coordinate (bottom) at (0,-7);
	\end{scope}
	
	\draw [BlueLine] (top.south) -- (a1.north);
	\draw [BlueLine] (top.south) -- (a2.north);
	\draw [GreenLine] (top.south) -- (a4.north);
	\draw [BlueLine] (top.south) --(a6.north);
		
	\draw [GreenLine] (a1.south) -- (b2.north);
	\draw [BlueLine] (a2.south) -- (b1.north);
	\draw [GreenLine] (a2.south) -- (b4.north);
	\draw [BlueLine] (a4.south) -- (b2.north);
	\draw [BlueLine] (a4.south) -- (b4.north);
	\draw [GreenLine] (a4.south) -- (b5.north);
	\draw [BlueLine] (a6.south) -- (b4.north);
		
	\draw [GreenLine] (b1.south) -- (bottom.north);
	\draw [GreenLine] (b2.south) -- (bottom.north);
	\draw [BlueLine] (b4.south) -- (bottom.north);
	\draw [BlueLine] (b5.south) -- (bottom.north);
	
	\begin{scope}
		\node[plate] (top) at (0,0) {};
		\foreach \i in {1,...,6} \node[plate] (a\i) at (-5.25+\i*1.5,-2) {};
		\foreach \i in {1,...,6} \node[plate] (b\i) at (-5.25+\i*1.5,-5) {};
		\node[plate] (bottom) at (0,-7) {};
	\end{scope}
	
	\begin{scope}[GrayPoly]
	\begin{scope}[shift={(0,0)}] \makepoints \filldraw (1) -- (2) -- (3) -- (4) -- cycle; \drawpoints \end{scope}
	
	\begin{scope}[shift={(-3.75,-2)}] \makepoints \filldraw (2) -- (3) -- (4)-- cycle; \drawpoints \end{scope}
	\begin{scope}[shift={(-2.25,-2)}] \makepoints \draw (1) -- (4); \draw (2) -- (3); \drawpoints \end{scope}
	\begin{scope}[shift={(-0.75,-2)}] \makepoints \filldraw (1) -- (2) -- (3) -- cycle; \drawpoints \end{scope}
	\begin{scope}[shift={(0.75,-2)}] \makepoints \filldraw (1) -- (2) -- (4) -- cycle; \drawpoints \end{scope}
	\begin{scope}[shift={(2.25,-2)}] \makepoints \draw (1) -- (2); \draw (3) -- (4); \drawpoints \end{scope}
	\begin{scope}[shift={(3.75,-2)}] \makepoints \filldraw (1) -- (3) -- (4) -- cycle; \drawpoints \end{scope}
	
	\begin{scope}[shift={(-3.75,-5)}] \makepoints \draw (2) -- (3); \drawpoints \end{scope}
	\begin{scope}[shift={(-2.25,-5)}] \makepoints \draw (2) -- (4); \drawpoints \end{scope}
	\begin{scope}[shift={(-0.75,-5)}] \makepoints \draw (3) -- (4); \drawpoints \end{scope}
	\begin{scope}[shift={(0.75,-5)}] \makepoints \draw (1) -- (4); \drawpoints \end{scope}
	\begin{scope}[shift={(2.25,-5)}] \makepoints \draw (1) -- (2); \drawpoints \end{scope}
	\begin{scope}[shift={(3.75,-5)}] \makepoints \draw (1) -- (3); \drawpoints \end{scope}
	
	\begin{scope}[shift={(0,-7)}] \makepoints \drawpoints \end{scope}
	\end{scope}
\end{tikzpicture}
\caption{The maximal chains which form $\pos(\PF_{3,3})$}
\end{figure}
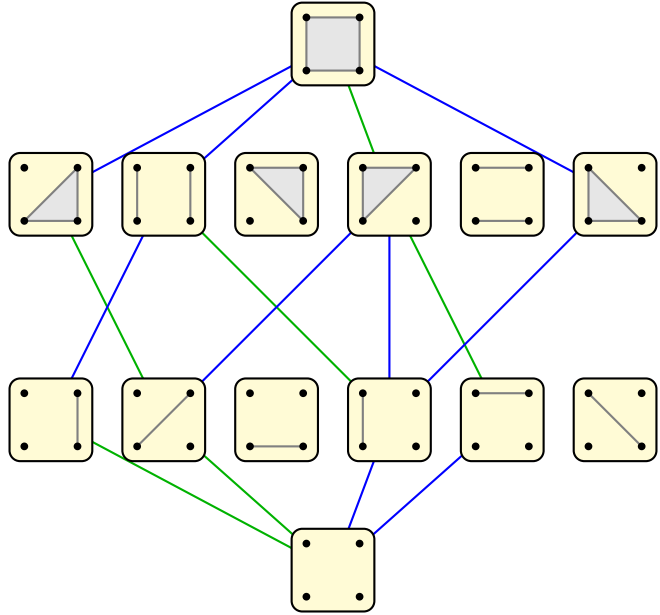

\newpage

\begin{prop}
  \label{pnn}
  Define two subsets of $\NC_{n+1}$ as follows:
  \[L_1 = \{\pi \mid \{n,n+1\}\in \pi\},\]
  \[L_2 = \{\pi \mid \{n+1\}\in \pi \text{ and } 1\sim n \text{ in } \pi\}.\]
  Then the elements in $\pos(\PF_{n,n})$ are precisely those in $\NC_{n+1} - (L_1 \cup L_2)$.
\end{prop}

Notice that $L_1 \cong L_2 \cong \NC_{n-1}$ since any $\pi \in \NC_{n-1}$ can be sent to an element of $L_1$ by adding 
$\{n,n+1\}$ as a block or to $L_2$ by adding $\{n+1\}$ as a block and inserting $n$ into the block containing $1$, and removing $n$ and $n+1$ reverses these inclusions. In particular, the proposition above 
tells us that 
\[|\pos(\PF_{n,n})| = C_{n+1} - 2C_{n-1}\]
since $L_1$ and $L_2$ are disjoint copies of $\NC_{n-1}$.
Similar to our result for $\PF_{n,k}$ and $\PF_{n-1,k}$, we can see that there is a recursive structure to these posets. In this setting, the corresponding result is that $\pos(\PF_{n,k})$ is the direct product of $\pos(\PF_{n-1,k})$ and a two-element chain. Applying this repeatedly, we obtain a useful result.

\begin{figure}
  \begin{tikzpicture}
    \begin{scope}[shift={(-3,0)},BluePoly]
      \node () [disk,minimum size=4.5cm] {};
      \foreach \n in {1,...,8} {
        \coordinate (\n) at (202.5+\n*-45:1.5cm) {};
      }
      \draw (7) -- (8);
      \draw (3) -- (4);
      \filldraw (1) -- (5) -- (6) -- cycle;
      
      \node[black] () at (202.5-45:1.8cm) {\small 1};
      \node[black] () at (202.5+10:1.8cm) {\small $n+1$};
      \node[black] () at (202.5+50:1.8cm) {\small $n$};
      \foreach \n in {1,...,8} {
        \draw (\n) node [dot] {};
      }
    \end{scope}

    \begin{scope}[shift={(3,0)},BluePoly]
      \node () [disk,minimum size=4.5cm] {};
      \foreach \n in {1,...,8} {
        \coordinate (\n) at (202.5+\n*-45:1.5cm) {};
      }
      \draw (3) -- (4);
      \filldraw (1) -- (5) -- (6) -- (7) -- cycle;
      \node[black] () at (202.5-45:1.8cm) {\small 1};
      \node[black] () at (202.5+10:1.8cm) {\small $n+1$};
      \node[black] () at (202.5+50:1.8cm) {\small $n$};
      \foreach \n in {1,...,8} {
        \draw (\n) node [dot] {};
      }
    \end{scope}
  \end{tikzpicture}
  \caption{The elements of $L_1$ and $L_2$ in $\NC_8$ which correspond to $\{\{1,5,6\},\{2\},\{3,4\}\}$
  in $\NC_6$}
\end{figure}
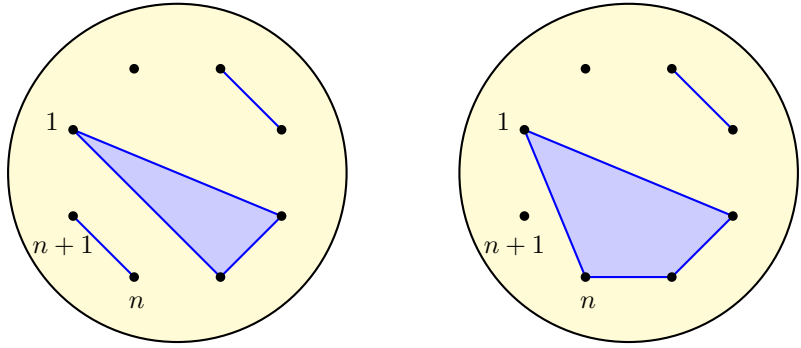

%
%

\begin{thm}
  \label{product}
  Suppose $n>k$ and let $B_{n-k}$ be the Boolean lattice of height $n-k$. Then
  $\pos(\PF_{n,k}) = \pos(\PF_{k,k}) \times B_{n-k}$.
\end{thm}

\begin{proof}
  Recall that $B_m$ is the direct product of $m$ Boolean lattices of height $1$, i.e. the direct product of $m$ 2-element chains. Hence, it suffices to show that $\pos(\PF_{n,k}) = \pos(\PF_{n-1,k}) \times B_1$.
  To do so, we construct an explicit decomposition of $\pos(\PF_{n,k})$ into two isomorphic copies of $\pos(\PF_{n-1,k})$ which are related appropriately in the poset structure. 

  Define two elements of $\pos(\PF_{n,k})$ as follows: \footnote{The elements $\sigma$ and $\tau$ do not exist in $\pos(\PF_{n,n})$
  but are present in each other $\pos(\PF_{n,k})$ for which $n>k$, which is the case here.}
  \[\sigma = \{\{1\},\ldots,\{n-1\},\{n,n+1\}\}\]
  \[\tau = \{\{1,\ldots, n\},\{n+1\}\}\]
  Let $\hat{0}$ and $\hat{1}$ denote the minimum and maximum elements of $\NC_{n+1}$, respectively, and 
  consider the intervals $[\sigma,\hat{1}]$ and $[\hat{0},\tau]$ in $\pos(\PF_{n,k})$ - these consist of elements in which 
  $n$ and $n+1$ share a block and in which $n+1$ forms a singleton block, respectively. Notice that 
  their (disjoint) union is $\pos(\PF_{n,k})$ since each element 
  in this 
  poset must be part of a maximal chain in which an $n$-label is created. Now, construct a map from 
  $\pos(\PF_{n,k})$ to $\pos(\PF_{n-1,k})$ which ``forgets" $n+1$ by removing it from whichever block it
   was in. This map respects the poset structure and is a bijection when restricted to either of the above
    intervals, hence each is isomorphic to $\pos(\PF_{n-1,k})$. Following these bijections gives us an
     isomorphism
  \[\Phi: [\hat{0},\tau] \to [\sigma,\hat{1}],\]
  and we can see that for each $\pi \in [\hat{0},\tau]$, $\pi \leq \Phi(\pi)$. Therefore, we have found two copies of $\pos(\PF_{n-1,k})$ which satisfy the structure of a direct product with a 2-element chain. That is,
  \[\pos(\PF_{n,k}) = \pos(\PF_{n-1,k}) \times B_1.\]
\end{proof}

Theorem \ref{product} tells us that we may focus our investigation on $\pos(\PF_{n,n})$ since many
desirable poset properties respect the operation of direct product. Interestingly, this poset cannot be
factored any further via the direct product - it is irreducible in this sense.

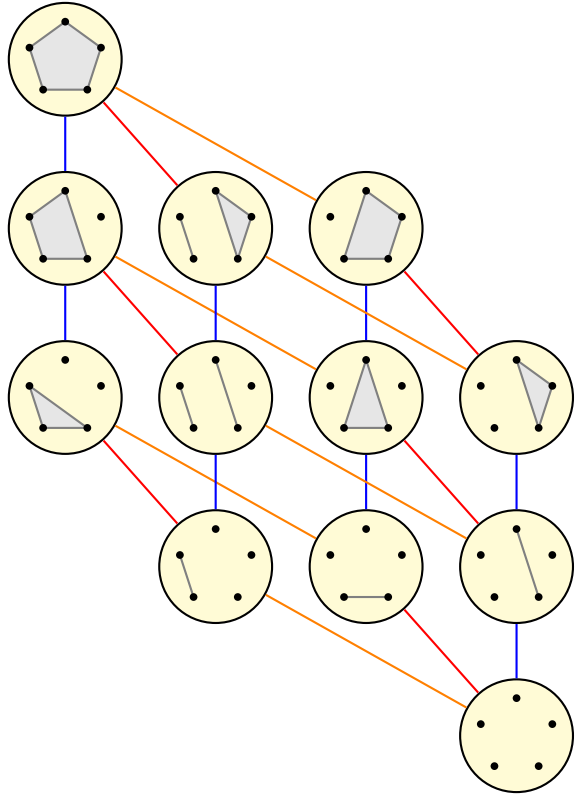
\begin{figure}
	\begin{tikzpicture}
		\node[disk,minimum size=1.5cm] (top) at (0,0) {};
		\foreach \i in {1,...,3} \node[disk,minimum size=1.5cm] (a\i) at (-2+\i*2,-2.25) {};
		\foreach \i in {1,...,4} \node[disk,minimum size=1.5cm] (b\i) at (-2+\i*2,-4.5) {};
		\foreach \i in {1,...,3} \node[disk,minimum size=1.5cm] (c\i) at (\i*2,-6.75) {};
		\node[disk,minimum size=1.5cm] (bottom) at (6,-9) {};
		
		\begin{scope}[GrayPoly]
			\begin{scope}[shift={(0,0)}] 
				\makepent 
				\filldraw (1) -- (2) -- (3) -- (4) -- (5) -- cycle; 
				\drawpent 
			\end{scope}
			
			\begin{scope}[shift={(0,-2.25)}] \makepent \filldraw (1) -- (3) -- (4) -- (5) -- cycle; \drawpent \end{scope}
			\begin{scope}[shift={(2,-2.25)}] \makepent \filldraw (1) -- (2) -- (3) -- cycle; \draw (4) -- (5); \drawpent \end{scope}
			\begin{scope}[shift={(4,-2.25)}] \makepent \filldraw (1) -- (2) -- (3) -- (4) -- cycle; \drawpent \end{scope}
			
			\begin{scope}[shift={(0,-4.5)}] \makepent \filldraw (3) -- (4) -- (5) -- cycle; \drawpent \end{scope}
			\begin{scope}[shift={(2,-4.5)}] \makepent \draw (1) -- (3); \draw (4) -- (5); \drawpent \end{scope}
			\begin{scope}[shift={(4,-4.5)}] \makepent \filldraw (1) -- (3) -- (4) -- cycle; \drawpent \end{scope}
			\begin{scope}[shift={(6,-4.5)}] \makepent \filldraw (1) -- (2) -- (3) -- cycle; \drawpent \end{scope}
			
			\begin{scope}[shift={(2,-6.75)}] \makepent \draw (4) -- (5); \drawpent \end{scope}
			\begin{scope}[shift={(4,-6.75)}] \makepent \draw (3) -- (4); \drawpent \end{scope}
			\begin{scope}[shift={(6,-6.75)}] \makepent \draw (1) -- (3); \drawpent \end{scope}
			
			\begin{scope}[shift={(6,-9)}] 
				\makepent
				\drawpent 
			\end{scope}
		\end{scope}
		
		\draw[OrangeLine] (top) -- (a3);
		\draw[BlueLine] (top) -- (a1);
		\draw[RedLine] (top) -- (a2);

		\draw[OrangeLine] (a1) -- (b3);
		\draw[BlueLine] (a1) -- (b1);
		\draw[BlueLine] (a2) -- (b2);
		\draw[BlueLine] (a3) -- (b3);
		\draw[RedLine] (a1) -- (b2);
		\draw[RedLine] (a3) -- (b4);	
		\draw[OrangeLine] (a2) -- (b4);	
		
		\draw[OrangeLine] (b1) -- (c2);	
		\draw[BlueLine] (b2) -- (c1);
		\draw[BlueLine] (b3) -- (c2);
		\draw[BlueLine] (b4) -- (c3);
		\draw[RedLine] (b1) -- (c1);
		\draw[RedLine] (b3) -- (c3);
		\draw[OrangeLine] (b2) -- (c3);	
		
		\draw[OrangeLine] (c1) -- (bottom);
		\draw[BlueLine] (c3) -- (bottom);
		\draw[RedLine] (c2) -- (bottom);	
	\end{tikzpicture}
	\caption{The direct product structure of $\pos(\PF_{4,2})$, where we label the top vertex of each noncrossing partition by $1$ and proceed clockwise.}
\end{figure}

%
%

\begin{prop}
  \label{irred}
  $\pos(\PF_{n,n})$ is irreducible as a direct product of posets.
\end{prop}

\begin{proof}
  Notice that if $P$ and $Q$ are posets of height $m$ and $n$, respectively, then there is a $1$-to-$\binom{m+n}{m}$ correspondence between pairs of maximal chains from $P$ and $Q$ and the set of maximal chains in $P\times Q$. We can see this by realizing that the maximal chains in the direct product of $P$ and $Q$  are formed by merging a maximal chain from each, and there are $\binom{m+n}{m}$ ways to do this.

  If $\pos(\PF_{n,n})$ were reducible, then the number of maximal chains in $\pos(\PF_{n,n})$ could be expressed as
  $\binom{n}{k}$ for some $k\in\{1,\ldots ,n-1\}$. However, we know that the number of maximal chains in $\pos(\PF_{n,n})$ is 
  \[|\PF_{n,n}| = (n+1)^{n-1} - n^{n-1},\]
  which is coprime to $n$, whereas $\binom{n}{k}$ shares a nontrivial factor with $n$ since 
  $\gcd(\binom{n}{i},\binom{n}{j}) > 1$ whenever $0 < i,j < n$. \cite{erdos} Therefore, $\pos(\PF_{n,n})$ is irreducible.
\end{proof}

%
%

In addition to the recursive structure depicted above, the symmetry
exhibited by these posets convinces us that they are a worthwhile setting for investigation.

\begin{thm}
  \label{selfdual}
  $\pos(\PF_{n,k})$ is self-dual.
\end{thm}

\begin{proof}
  There are $2n+2$ order-reversing automorphisms of $\NC_{n+1}$ \cite{biane97}, so we will describe one and show that it fixes $\pos(\PF_{n,n})$. Self-duality of $\pos(\PF_{n,k})$ then follows from Theorem \ref{product} since the product of self-dual posets is self-dual.

  Let $\pi \in \NC_{n+1}$ and suppose $i\leq j$ for some $i,j\in [n+1]$. Define
  \[A_{i,j} = \{n-j+1, n-j+2, \ldots , n-i-1, n-i\}\]
  and 
  \[B_{i,j} = [n+1] - A_{i,j}.\]
  We then define a map $\rho : \NC_{n+1} \to \NC_{n+1}$ by declaring that $i\sim j$ in $\rho(\pi)$ if and only if $x\not\sim y$ in $\pi$ for all $x\in A_{i,j}$ and $y\in B_{i,j}$.

  \begin{figure}
    \begin{tikzpicture}

      \begin{scope}[shift={(-4.4,0)},BluePoly]
	\node () [disk,minimum size=3.6cm] {};
	\foreach \n in {1,...,8} {
          \coordinate (\n) at (202.5+\n*-45:1.2cm) {};
          \node[black] () at (202.5+\n*-45:1.5cm) {\n};
	}
	\filldraw (1) -- (3) -- (8) -- cycle;
	\filldraw (5) --(6) -- (7) -- cycle;
	\foreach \n in {1,...,8} {
          \draw (\n) node [dot] {};
	}
      \end{scope}

      \draw[->] (-2.4,0) to (-2,0);

      \begin{scope}[shift={(0,0)}]
	\node () [disk,minimum size=3.6cm] {};
	\foreach \n in {1,...,8} {
          \coordinate (\n) at (202.5+\n*-45:1.2cm) {};
          \coordinate (\n+8) at (-135+\n*45:1.2cm) {};
          \node[black] () at (202.5+\n*-45:1.5cm) {\n};
          \node[gray] () at (-135+\n*45:1.5cm) {\n};

	}
	\filldraw[BluePoly] (1) -- (3) -- (8) -- cycle;
	\filldraw[BluePoly] (5) -- (6) -- (7) -- cycle;
	\filldraw[RedPoly,dashed] (3+8) -- (4+8) -- (8+8) -- cycle; 
	\draw[RedPoly,dashed] (5+8) -- (6+8);
	\foreach \n in {1,...,8} {
          \fill (\n) circle (2pt);
          \draw[color=black,fill=white] (\n+8) circle (2pt);
	}
      \end{scope}

      \draw[->] (2,0) to (2.4,0);

      \begin{scope}[shift={(4.4,0)},RedPoly]
	\node () [disk,minimum size=3.6cm] {};
	\foreach \n in {1,...,8} {
          \coordinate (\n) at (202.5+\n*-45:1.2cm) {};
          \node[black] () at (202.5+\n*-45:1.5cm) {\n};
	}
	\filldraw[RedPoly] (3) -- (4) -- (8) -- cycle;
	\draw[red] (5) -- (6);
	\foreach \n in {1,...,8} {
          \draw (\n) node [dot] {};
	}
      \end{scope}
    \end{tikzpicture}
    \caption{An order-reversing involution of $\NC_8$}
    \label{dual}
  \end{figure}
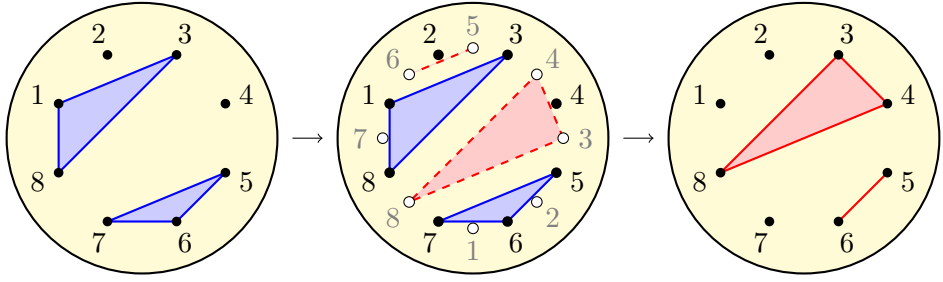

  Put another way, we can obtain $\rho(\pi)$ by first drawing the circular representation of $\pi$ and interspersing $n+1$ extra white points between the preexisting black points so that, proceeding clockwise, the black point labeled $i$ is just before the white point labeled $n-i$, modulo $n+1$. Then, to obtain the noncrossing partition $\rho(\pi)$, we take the convex hulls of white points which do not cross the original partition - see Figure \ref{dual}.

  It is easy to see that this forms an order-reversing involution on $\NC_{n+1}$ and is thus an isomorphism. We need to check that this map fixes the elements $\pos(\PF_{n,n})$ and is order-reversing in that poset's relation.

  Notice that if $\{n,n+1\} \in \pi$, then $\{n+1\} \in \rho(\pi)$ and $1 \sim n$ in $\rho(\pi)$, so $\rho$ restricts to an
  order-reversing isomorphism $L_1 \to L_2$ as described in Proposition \ref{pnn}. Hence $L_1\cup L_2$ is
  fixed by $\rho$ and thus so is $\pos(\PF_{n,n})$.
  
  To see that $\rho$ is order-reversing with respect to $\pos(\PF_{n,n})$, it suffices to show that any covering relation in $\NC_{n+1}$ which produces an $n$-label is sent to another $n$-labeled covering relation. Thankfully, we know that such a covering relation $\sigma < \tau$ occurs if and only if $\{n+1\} \in \sigma$ and $n \sim n+1$ in $\tau$. Examining the definition of $\rho$, we see this is equivalent to knowing that $n\sim n+1$ in $\rho(\sigma)$ and $\{n+1\} \in \rho(\tau)$, so the new covering relation is also labeled by an $n$. Hence, $\rho$ is an order-reversing automorphism of $\NC_{n+1}$ which fixes $\pos(\PF_{n,n})$ and respects its poset structure, so $\rho$ is an order-reversing automorphism of $\pos(\PF_{n,n})$.
\end{proof}

The rest of our results make use of the convenient structure of intervals in $\pos(\PF_{n,n})$, on which we must first prove a somewhat technical theorem, mirroring the fact that each interval in $\NC_{n+1}$ is a direct product of smaller noncrossing
partition lattices.

%
%

\begin{thma}
  \label{intervals}
  Each interval in $\pos(\PF_{n,n})$ is either of the form 
  \[\prod_{i=1}^j \NC_{m_i} \text{ or } \pos(\PF_{r,r}) \times \prod_{i=1}^j \NC_{m_i}.\]
\end{thma}

\begin{proof}
  Let $\sigma, \tau \in \pos(\PF_{n,n})$ with $\sigma < \tau$. If $\sigma = \hat{0}$ and $\tau = \hat{1}$, then $[\sigma,\tau] =\pos(\PF_{n,n})$ and we're done. For now, suppose that $\tau \neq \hat{1}$. Recall that the only way to produce an $n$-label in a chain between two elements in $\NC_{n+1}$ is if $n\sim n+1$ in the coarser partition and $n+1$ is a singleton in the finer partition. If one of these conditions is not satisfied, i.e. $n \not\sim_\tau n+1$ or $\{n+1\} \not\in \sigma$, then
  \[[\sigma,\tau]_{\pos(\PF_{n,n})} = [\sigma,\tau]_{\NC_{n+1}}\]
  and since we know that intervals in $\NC_{n+1}$ are products of smaller noncrossing partition lattices \cite{nica}, 
  the result follows for this case.

  Now, suppose that $n\sim_\tau n+1$ and $\{n+1\}\in \sigma$. Then there are elements between $\sigma$ and $\tau$ in $\NC_{n+1}$ which do not appear in $\pos(\PF_{n,n})$, so we need to try something different. 

  Let $B \in \tau$ be the block containing $n$ and $n+1$ and let $B_1,\ldots B_k\in \sigma$ be the blocks whose union is $B$. If we write $B = \{b_1,\ldots ,b_{l+1}\}$ such that $b_1 < \ldots <b_{l+1}$, then we can ``factor out" an interval in $\pos(\PF_{l,l})$ from $[\sigma, \tau]$ in the following way: define $f:B\to [l+1]$ by $f(b_i)=i$ and notice that 
  \[\sigma^\prime = \{f(B_1),\ldots f(B_k)\} \in \NC_{l+1}.\]
  In fact, $\sigma^\prime \in \pos(\PF_{l,l})$ by Proposition \ref{pnn}, which we will use momentarily.

  Next, define
  \[\tau^\prime = \{\tau - \{B\}\} \cup \{B_1,\ldots B_k\}.\]
  Notice that $\sigma$ and $\tau^\prime$ have blocks which agree on the elements of $B$. The benefit of making this modification is that the interval $[\sigma,\tau^\prime]_{\pos(\PF_{n,n})}$ is equal to $[\sigma,\tau^\prime]_{\NC_{n+1}}$ since we have ``removed" the blocks which could create $n$-labels. Similarly, the interval $[\sigma^\prime,\hat{1}]_{\pos(\PF_{l,l})}$ contains (an isomorphic copy of) only those elements. More concretely, we have decomposed our interval to obtain
  \[[\sigma,\tau]_{\pos(\PF_{n,n})} = [\sigma^\prime,\hat{1}]_{\pos(\PF_{l,l})} \times [\sigma,\tau^\prime]_{\NC_{n+1}}.\]
  Applying our order-reversing map $\rho$, we can see that
  \[[\sigma^\prime,\hat{1}] \cong \rho([\sigma^\prime,\hat{1}]) = [\rho(\hat{1}),\rho(\sigma^\prime)] = [\hat{0},\rho(\sigma^\prime)]\]
  in $\pos(\PF_{l,l})$. Notice that since $\{n+1\} \in \sigma$, then $\{l+1\} \in \sigma^\prime$, so we have $l \sim l+1$ in $\rho(\sigma^\prime)$. Then the interval $[\hat{0},\rho(\sigma^\prime)]$ decomposes into the direct product of $\pos(\PF_{r,r})$ (where $r$ is the size of the block containing $l$ and $l+1$ in $\rho(\sigma^\prime)$) and $\prod_{i=1}^j \NC_{m_i}$ (where the $m_i$'s are the sizes of the other blocks in $\rho(\sigma^\prime)$). Hence, we have
  \begin{align*}
    [\sigma,\tau]_{\pos(\PF_{n,n})} &\cong [\hat{0},\rho(\sigma^\prime)]_{\pos(\PF_{l,l})} \times [\sigma, \tau^\prime]_{\NC_{n+1}} \\
    &\cong \pos(\PF_{r,r}) \times \prod_{i=1}^j \NC_{m_i} \times [\sigma, \tau^\prime]_{\NC_{n+1}}
  \end{align*}
  Since we know the form of intervals in $\NC_{n+1}$, we're done.
\end{proof}

With this result in hand, we can prove a few more properties about the structure of these posets. In particular, we have proven a stronger version of Theorem \ref{selfdual}.

%
%

\begin{thma}
  $\pos(\PF_{n,k})$ is locally self-dual.
\end{thma}

We know that $\pos(\PF_{n,n})$ is self-dual via our map $\rho$, but it is worth wondering whether there are other maps which
would work just as well. Interestingly, this symmetry is unique - put another way, there is only one order-preserving automorphism.

%
%

\begin{thm}
  The identity map is the unique order-preserving automorphism on $\pos(\PF_{n,n})$.
\end{thm}

\begin{proof}
Let $f$ be an order-preserving automorphism of $\pos(\PF_{n,n})$ and consider the two elements characterized as an edge from 
$1$ to $n+1$ and the triangle between $1$, $n$, and $n+1$:
\[E \coloneqq \{\{1,n+1\},\{2\},\ldots , \{n\}\}\]
and
\[T \coloneqq \{\{1,n,n+1\},\{2\},\ldots, \{n-1\}\}.\]
Although $T$ has three children in $NC_{n+1}$, it is easy to see that it is the unique element at height 2 in $\pos(\PF_{n,n})$
with only one child: $E$. Hence, $T$ (and thus $E$ as well) must be fixed by $f$, and by duality via $\rho$, so must $\rho(T)$ and
$\rho(E)$. Then we know that the intervals
\[[E,\hat{1}] = \{\pi\in\pos(\PF_{n,n}) \mid 1 \sim n+1 \text{ in } \pi\}\]
and
\[[\hat{0},\rho(E)] = \{\pi \in \pos(\PF_{n,n}) \mid \{n\}\in \pi\}\]
are fixed setwise by $f$.

In fact, we can conclude something stronger - notice that $[E,\hat{1}]$ is isomorphic to $NC_n$ by merging $1$ with $n+1$. Since 
the automorphisms of $NC_n$ are the $2n$ natural dihedral symmetries \cite{biane97}, $f$ must restrict to one of these on this 
interval. But $E$ and $\rho(E)$ are fixed by $f$ and correspond to the elements
$\{\{1,n\},\{2,\ldots ,n-1\}\}$ and $\{\{1,\ldots ,n-1\},\{n\}\}$
in $NC_n$, so the identity is the only possibility. So $f$ fixes each element in $[E,\hat{1}]$ (and $[\hat{0},\rho(E)]$ by similar 
argument). 

All that remains is to show that elements outside these two intervals are fixed by $f$. Let $\sigma \in \pos(\PF_{n,n})$ be such
an element - then $\sigma$ covers $\sigma \wedge \rho(E)$  by splitting off $\{n\}$ and $\sigma \vee E$ covers $\sigma$ by joining 
the blocks containing $1$ and $n+1$. Since $\sigma\wedge\rho(E)$ and $\sigma \vee E$ are each fixed by $f$, the interval
$[\sigma \wedge \rho(E) , \sigma \vee E]$ is fixed setwise. If we can show that $\sigma$ is the unique element in this interval which
does not also lie in $[E,\hat{1}] \cup [\hat{0},\rho(E)]$, then we can conclude that $\sigma$ is fixed and we're done.

Suppose first that $n \sim n+1$ in $\sigma$. Then $1$, $n$, and $n+1$ lie in separate blocks in $\sigma \wedge \rho(E)$ and
share the same block in $\sigma \vee E$. The only way to obtain an intermediate element is to combine two of the blocks in 
$\sigma \wedge \rho(E)$; we obtain $\sigma$ by combining the blocks with $n$ and $n+1$. If we were to merge the $1$ and 
$n+1$ blocks (leaving $n$ a singleton) we would obtain an element in $[E,\hat{1}]\cap[\hat{0},\rho(E)]$, and if we merge the $1$
block with the $n$ block, the result is an element not in $\pos(\PF_{n,n})$. Hence, in this case the interval 
$[\sigma \wedge \rho(E) , \sigma \vee E]$ has two elements of height 1, each of which must be fixed.

Now suppose $n\not\sim n+1$ in $\sigma$. Then $1$, $n$, and $n+1$ live in distinct blocks in $\sigma \wedge \rho(E)$ as above,
but $1$ and $n+1$ share a block without $n$ in $\sigma \vee E$. Then there are only two possibilities for intermediate elements
in $[\sigma \wedge \rho(E) , \sigma \vee E]$, obtained by either merging $\{n\}$ with the rest of its block in $\sigma \vee E$ or by
combining the $1$ and $n+1$ blocks: the former results in $\sigma$ while the latter gives an element in 
$[E,\hat{1}]\cap[\hat{0},\rho(E)]$, so once again $\sigma$ is the unique element in this interval which lies outside of $[E,\hat{1}]$
and $[\hat{0},\rho(E)]$.

Therefore, we can see that $f$ must fix each element in $\pos(\PF_{n,n})$, and we're done.
\end{proof}

%
%

In the spirit of analyzing the symmetries and recursive structure of these posets, we also compute the M\"obius function of $\pos(\PF_{n,k})$.

\begin{defin}
Let $P$ be a poset. The \emph{M\"obius function} of $P$ is a map $\mu_P: P\times P \to \mathbb{Z}$ defined recursively as follows:
\[\mu_P(x,y) = 
\begin{cases} 
\hfil 1 & \text{if }x=y \\ 
\hfil 0 & \text{if }x>y \\ 
\displaystyle-\sum_{x\leq z < y}\mu_P(x,z) & \text{if } x < y 
\end{cases}\]
If $P$ has minimum and maximum elements $\hat{0}$ and $\hat{1}$ respectively, we simplify our notation by writing $\mu(P) \coloneqq \mu_P(\hat{0},\hat{1})$. If $P$ and $Q$ are two such posets, then we have the following properties:
\begin{enumerate}
\item $\mu(P\times Q) = \mu(P)\mu(Q)$
\item If $\phi : P \to Q$ is an order-reversing isomorphism, then we have $\mu_P(x,y) = \mu_Q(\phi(y),\phi(x))$.
\item $\displaystyle \sum_{\pi\in P} \mu_P(\hat{0},\pi) = 0$
\end{enumerate}
\end{defin}

As an example, the M\"obius function of $\NC_{n+1}$ is a pleasing computation which can be found in Kreweras'
 article introducing noncrossing partitions. \cite{kreweras}

\begin{prop}
  \label{mobius}
  The M\"obius function of $\NC_{n+1}$ is $(-1)^{n-1}C_n$.
\end{prop}

In the case of $\pos(\PF_{n,k})$, we find a similarly interesting result.

\begin{thma}
The M\"obius function of $\pos(\PF_{n,k})$ is zero.
\end{thma}
\begin{proof}
Since the M\"obius function respects direct products, it suffices to prove this for $\pos(\PF_{n,n})$; the first case is straightforward since $\pos(\PF_{2,2})$ is simply a 3-element chain.

Suppose $\mu(\pos(\PF_{k,k})) = 0$ for all $k<n$. For ease of notation, we define
\[P_{n,n} \coloneqq \pos(\PF_{n,n})\]
for this proof only. By definition, we have
\[\mu(P_{n,n}) = -\sum_{\pi < \hat{1}} \mu_{P_{n,n}}(\hat{0},\pi).\]
Now, if $\{n+1\} \not\in \pi$, then $n \not\sim n+1$ in $\rho(\pi)$ and by Theorem \ref{intervals}, we know that
$[\rho(\pi),\hat{1}]$ contains a factor of $\pos(\PF_{r,r})$ for some $r<n$, so $\mu([\rho(\pi),\hat{1}]) = 0$. Since $\rho$ is an order-reversing isomorphism, we can then compute
\[\mu_{P_{n,n}}(\hat{0},\pi) = \mu_{P_{n,n}}(\rho(\pi),\hat{1}) = 0,\]
so it suffices to compute the M\"obius function for elements of $P_{n,n}$ which contain $\{n+1\}$. 

Applying Theorem \ref{intervals} again, we see that $\mu_{P_{n,n}}(\hat{0},\pi) = \mu_{\NC_{n+1}}(\hat{0},\pi)$ when $\pi$ contains the singleton $\{n+1\}$, allowing us to work over $\NC_{n+1}$:
\begin{align*}
\mu_{P_{n,n}}(\hat{0},\hat{1}) &= -\sum_{\mathclap{\substack{\pi\in P_{n,n} \\ \{n+1\} \in \pi}}} \mu_{P_{n,n}}(\hat{0},\pi) \\
&= -\sum_{\mathclap{\substack{\pi \in P_{n,n} \\ \{n+1\}\in \pi}}} \mu_{\NC_{n+1}} (\hat{0},\pi) \\
&= -\sum_{\mathclap{\substack{\pi \in \NC_{n+1} \\ \{n+1\}\in \pi}}} \mu_{\NC_{n+1}} (\hat{0},\pi)
+ \sum_{\mathclap{\substack{\pi \in \NC_{n+1} \\ \{n+1\}\in \pi \\ 1\sim_\pi n}}} \mu_{\NC_{n+1}} (\hat{0},\pi)
\end{align*}
The terms on the right-hand side are sums of M\"obius functions over the entirety of two posets (isomorphic to $\NC_n$ and $\NC_{n-1}$ respectively), so each is zero and we're done.
\end{proof}

%
%

\begin{defin}
  A poset $P$ with height $n$ is said to have a \emph{symmetric chain decomposition} (SCD) if the elements of $P$ can be partitioned into saturated chains so that ranks of the minimum and maximum in each chain sum to $n$.
\end{defin}

It follows quickly from this definition that if posets $P$ and $Q$ have a symmetric chain decomposition, then so 
does $P\times Q$. 

\begin{thm}
  $\pos(\PF_{n,k})$ admits a symmetric chain decomposition.
\end{thm}

\begin{proof}
  By the observation above, it suffices to show that $\pos(\PF_{n,n})$ has an SCD, and we will do so by mirroring Simion and 
  Ullman's proof for $\NC_n$. \cite{simion} 
  
  First, notice that we can easily find SCDs when $n$ is 1,2, or 3 by glancing at the Hasse diagram. Now, suppose that this result 
  holds for values less than some fixed $n>3$ and decompose $\pos(\PF_{n,n})$ as follows:
  \begin{align*}
    R_1 &= [\hat{0},\{\{1\},\{2,\ldots ,n+1\}\}] \\
    R_1^\prime &= [\{\{1\},\{2,n\},\{3\},\ldots ,\{n-1\},\{n+1\}\},\{\{1\},\{2,\ldots ,n\},\{n+1\}\}] \\
    R_i &= \{\pi \mid i = \min\{j \mid i\sim j, j\neq 1\}\} \text{ for each } i \in\{2,\ldots ,n+1\}
  \end{align*}
  That is, $R_i$ is the subposet consisting of elements for which $i$ is the second-smallest number in the block containing 1.
  We claim that $\pos(\PF_{n,n})$ is the disjoint union of these subposets.
  
  It is straightforward to see from the definition that each $R_i$ is disjoint from $R_1$, $R_1^\prime$, and each other $R_j$. 
  Observing that $R_1 \cap R_1^\prime = \emptyset$ amounts to noticing that no element of $R_1$ could be less than the 
  maximum element in $R_1^\prime$. Hence, this is a disjoint collection of subposets for $\pos(\PF_{n,n})$.
  
  As for their union, notice that for any $\pi \in \pos(\PF_{n,n})$, there are three cases. If $\{1\} \not\in \pi$, then $\pi \in R_i$ for some
  $i \neq 1$. If not, then we either have $2\sim n$ in $\pi$, in which case $\pi \in R_1^\prime$, or $2\not\sim n$ and thus $\pi\in R_1$.
  So $\pos(\PF_{n,n})$ is the disjoint union of these subposets.
  
  Observing each $R_i$ as an interval for $i \in \{3,\ldots ,n-1\}$, we can see that each is isomorphic to a direct product of $\pos(\PF_{l,l})$ and/or $\NC_m$ 
  for values of $l$ and $m$ less than $n$, so each has a symmetric chain decomposition by our inductive hypothesis and the
  analogous result for $\NC_n$. \cite{simion} Additionally, each $R_i$ ranges from height 1 to height $n-1$ in $\pos(\PF_{n,n})$,
  so it is ``symmetrically embedded". Additionally, $R_1$ and $R_2$ are each isomorphic to $\pos(\PF_{n-1,n-1})$ and embedded
  so that $R_2$ covers $R_1$ - that is, $R_1\cup R_2$ is the direct product of a 2-element chain and $\pos(\PF_{n-1,n-1})$.
  Hence it admits an SCD which is also symmetrically embedded from height 0 to height $n$.
  
  It then remains to show that $R_1^\prime$, $R_n$, and $R_{n+1}$ together admit a symmetrically embedded SCD. To this end,
  define two subposets of $R_{n+1}$ as follows:
  \begin{align*}
    A &= [\{\{1,n+1\},\{2,n\},\{3\},\ldots ,\{n-1\}\},\{\{1,n+1\},\{2,\ldots ,n\}\}] \\
    B &= [\{\{1,n+1\},\{2\}, \ldots ,\{n\}\},\{\{1,n+1\},\{2,\ldots ,n-1\},\{n\}\}]
  \end{align*}
  Then we can see that $A\cong R_1^\prime$, $B\cong R_n$, and each is isomorphic to $\NC_{n-2}$. Moreover, $A$ coverse
  $R_1^\prime$ and $R_n$ covers $B$ in such a way that $A\cup R_1^\prime$ and $B\cup R_n$ are each isomorphic to the
  direct product $B_2 \times \NC_{n-2}$. Thus, each has an SCD and ranges from height 1 to $n-1$, so they are symmetrically
  embedded.
  
  All that remains is $R_{n+1} - (A\cup B)$, which can be expressed as
  \[R_{n+1} - (A\cup B) = \{ \pi \mid \{1,n+1\} \in \pi, \{n\}\not\in\pi, 2\not\sim n \text{ in } \pi\}.\]
  Similar to how we began, we can partition this subposet into the intervals 
  \[D_j = \{\pi \in R_{n+1} - (A\cup B) \mid j = \min\{k \mid n\sim k\}\},\]
  each of which is isomorphic to a direct product of noncrossing partition lattices. Therefore, each has an SCD which ranges from
  height 2 to heignt $n-2$.
  
  Since we have accounted for all subposets in our decomposition, we can conclude that $\pos(\PF_{n,n})$ admits a symmetric
  chain decomposition.
\end{proof}


\section{Order Complexes}

The study of order complexes gives us a topological way to understand $\pos(\PF_{n,k})$, possibly leading to greater insight for 
this decomposition.

\begin{defin}
  Let $P$ be a poset and define the \emph{order complex} $\Delta(P)$ to be the simplicial complex constructed by associating a $k$-simplex to each finite chain $x_0 < x_1 < \cdots < x_k$ in $P$ in the natural way.
\end{defin}

Notice that maximal chains in $P$ correspond to top-dimensional simplices in $\Delta(P)$, which thus can be labeled by parking functions in the case when $P$ is $\NC_{n+1}$. In particular, our decomposition of $\PF_n$ produces a decomposition of the facets of $\Delta(\NC_{n+1})$ and we can view our order-reversing map $\rho$ as a symmetry of this topological space.

In general, when $P$ is bounded we refer to the edge in the order complex which corresponds to the chain $\hat{0} < \hat{1}$ by the \emph{diagonal}. Notice that then $\Delta(P)$ is contractible since each of $\hat{0}$ and $\hat{1}$ is a cone point, so we may deformation retract the complex to either one. In order to see the combinatorial data of $P$ topologically, we disregard $\hat{0}$ and $\hat{1}$ and examine $\Delta(\overline{P})$, where 
$\overline{P} = P - \{\hat{0},\hat{1}\}$.

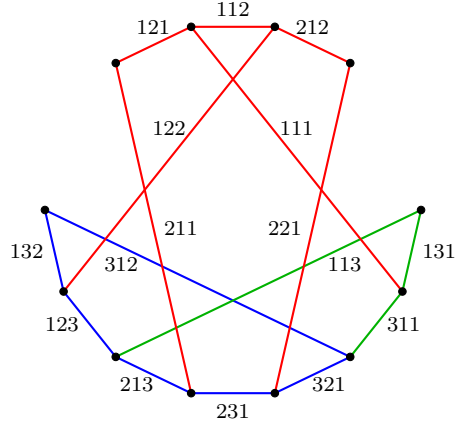
\begin{figure}
	\begin{tikzpicture}
		\foreach \n in {2,...,5} {
			\coordinate (a\n) at (180-\n*360/14:2.5cm) {};
		}
		\foreach \n in {0,...,7} {
			\coordinate (b\n) at (-\n*360/14:2.5cm) {};
		}
		
		\begin{scope}[midway]
			\draw[BlueLine] 
			(b2) -- node[near end, right=3pt]{\scriptsize 321}
			(b3) -- node[below]{\scriptsize 231}
			(b4) -- node[near start,left=3pt]{\scriptsize 213}
			(b5) -- node[left]{\scriptsize 123}
			(b6) -- node[left]{\scriptsize 132}
			(b7) -- node[near start,below]{\scriptsize 312} (b2);
				
			\draw[GreenLine] 
			(b5) -- node[near end, below]{\scriptsize 113}
			(b0) -- node[right]{\scriptsize 131}
			(b1) -- node[right]{\scriptsize 311} (b2);
			
			\draw[RedLine] 
			(b4) -- node[right]{\scriptsize 211} 
			(a2) -- node[above]{\scriptsize 121} 
			(a3) -- node[above]{\scriptsize 112}
			(a4) -- node[above]{\scriptsize 212}
			(a5) -- node[left]{\scriptsize 221} (b3);
			\draw[RedLine] (b6) -- node[above=5pt]{\scriptsize 122} (a4);
			\draw[RedLine] (b1) -- node[above=5pt]{\scriptsize 111} (a3);
		\end{scope}
		
		\foreach \n in {2,...,5} {
			\draw (a\n) node [dot] {};
		}
		\foreach \n in {0,...,7} {
			\draw (b\n) node [dot] {};
		}
	\end{tikzpicture}
	\caption{The link of $\NC_4$ decomposed into the links of $\pos(\sym_3)$, $\pos(\PF_{3,2})$, and $\pos(\PF_{3,3})$. The action of our unique order-reversing map $\rho$ on the maximal chains of $\NC_4$ can be realized geometrically as a reflection through the vertical axis.}
\end{figure}

It turns out that $\Delta(\overline{P})$ has another topological name - it is the link of the diagonal in $\Delta(P)$. As such, we refer to $\Delta(\overline{P})$ as the \emph{link complex} or simply \emph{link} of $P$. There is an intimate connection between the topology of this complex and the combinatorics of $P$.

\begin{thm}[Philip Hall's Theorem]
  Let $P$ be a bounded poset. Then
  \[\mu(P) = \widetilde{\chi}(\Delta(\overline{P})),\]
  where $\widetilde{\chi}$ is the reduced Euler characteristic.
\end{thm}

Hence, by Theorem \ref{mobius}, we know that the link of $\pos(\PF_{n,k})$ has reduced Euler characteristic $0$. This, in combination with low-dimensional computations, gives compelling evidence for the following conjecture.

\begin{conj}
  The link of $\pos(\PF_{n,k})$ is contractible.
\end{conj}


\section{Other Settings}

While $\NC_{n+1}$ is an exceptionally fruitful setting for exploring our decomposition, there are other objects in bijection with $\PF_n$ with matching decompositions. Specifically, there are geometric ways to visualize $\PF_{n,k}$ with labeled Dyck paths and labeled rooted forests. If nothing else, these examples speak for the natural definition of $\PF_{n,k}$.

\begin{defin}
  A \emph{Dyck path} of length $2n$ is a lattice path in $\mathbb{Z}\oplus \mathbb{Z}$ from $(0,0)$ to $(n,n)$ which stays weakly above the diagonal $y=x$ and consists only of ``up" and ``right" steps. A \emph{labeled Dyck path} is a Dyck path for which the $n$ ``up" steps are distinctly labeled by the $\{1,\ldots ,n\}$.
\end{defin}

\begin{prop}
  \label{dyck}
  There is a one-to-one correspondence between labeled Dyck paths of length $2n$ and $\PF_n$.
\end{prop}

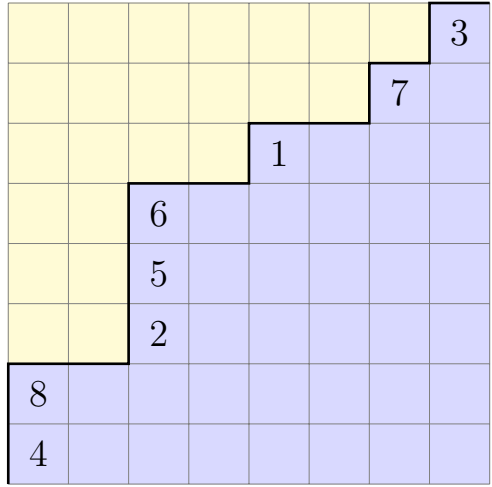
\begin{figure}
	\begin{tikzpicture}[scale=0.8]
		\draw [fill=yellow!20,yellow!20] (0,0) rectangle (8,8);
		\draw [fill=blue!15,blue!15] (0,0) rectangle (8,2);
		\draw [fill=blue!15,blue!15] (2,2) rectangle (8,5);
		\draw [fill=blue!15,blue!15] (4,5) rectangle (8,6);
		\draw [fill=blue!15,blue!15] (6,6) rectangle (8,7);
		\draw [fill=blue!15,blue!15] (7,7) rectangle (8,8);
		
		\draw [help lines] (0,0) grid (8,8);
		\draw[line width=1pt] 
		(0,0) -- (0,1) -- (0,2) -- (1,2) -- (2,2) -- (2,3) -- 
		(2,4) -- (2,5) -- (3,5) -- (4,5) -- (4,6) -- 
		(5,6) -- (6,6) -- (6,7) -- (7,7) -- (7,8) -- (8,8);
		\node () at (0.5,0.5) {\Large 4};
		\node () at (0.5,1.5) {\Large 8};
		\node () at (2.5,2.5) {\Large 2};
		\node () at (2.5,3.5) {\Large 5};
		\node () at (2.5,4.5) {\Large 6};
		\node () at (4.5,5.5) {\Large 1};
		\node () at (6.5,6.5) {\Large 7};
		\node () at (7.5,7.5) {\Large 3};
	
	\end{tikzpicture}
	\caption{The labeled Dyck path corresponding to the parking function $(5,3,8,1,3,3,7,1) \in \pos(\PF_{8,6})$}
\end{figure}

Given $(a_1,\ldots ,a_n) \in \PF_n$, let $b_i$ be the number of times $i$ appears in this parking function and
draw a path from $(0,0)$ to $(n,n)$ with $b_1$ steps up followed by one step to the right, then $b_2$ steps up followed by one step to the right, and so on. The restrictions on parking functions ensure that this path will stay weakly above the diagonal. Label this Dyck path as follows: label the bottom-most unlabeled ``up" step in column $a_1$ with 1. Label the same such step in column $a_2$ with a 2, and so on. This results in a labeled Dyck path, and this map is a bijection.

\begin{defin}
  Let $D_{n,k}$ be the set of labeled Dyck paths which correspond to the parking functions in $\PF_{n,k}$ via the bijection above,
  and let $D_{n,1}$ be those which correspond to parking functions identified with $\sym_n$.
\end{defin}

It follows immediately from the bijection that if $k$ is the largest number missing from a parking function $f$, then the corresponding labeled Dyck path has no up-step in the $k$th column, followed by $n-k$ columns each with exactly one up-step. In other words, we have characterized this decomposition with the following theorem.

\begin{thm}
  $D_{n,k}$ is the set of labeled Dyck paths of length $2n$ which end in a string of exactly $n-k$ ``up-right" pairs.
\end{thm}

There is a similar way to see the $\PF_{n,k}$ decomposition in the setting of rooted forests.

\begin{defin}
  A \emph{rooted forest} is a collection of trees in which each tree has one vertex designated as the root. A \emph{labeled rooted forest} is one in which each vertex is labeled with a distinct element of $\{1,2,\dots, n\}$, where $n$ is the total number of vertices.  We denote the set of labeled rooted forests on $n$ vertices as 
  $\RF_n$.
\end{defin}

There is a bijection from $\RF_n$ to $\PF_n$ that extends our decomposition to this setting in a natural way.
\cite{Pak}

\begin{defin}
  Let $v$ and $w$ be vertices in a rooted forest $f$. 
  \begin{enumerate}
  \item $h(v)$ is the length of the shortest path between $v$ and a root.
  \item If $v$ and $w$ are adjacent and $h(w) = h(v)+1$, then $v$ is the \emph{parent} of $w$ and $w$ is a
  \emph{child} of $v$.
  \item Two vertices are \emph{siblings} if they either share the same parent or are both roots.
  \item $\eta(v) = i$ if $v$ has the $i^{th}$ smallest label among its siblings. 
  \end{enumerate}
\end{defin}

\begin{figure}
	\begin{tikzpicture}
		\begin{scope}[shift={(-3.5,0)}]
			\coordinate (2) at (-0.75,0) {};
			\coordinate (1) at (-0.75,-1) {};
			\coordinate (3) at (0.75,0) {};
			\coordinate (4) at (0.75,-1) {};
			\coordinate (5) at (0,-2) {};
			\coordinate (6) at (0.75,-2) {};
			\coordinate (7) at (1.5,-2) {};
			\foreach \n in {1,...,4} {
				\draw (\n) node [dot,label=left:\n] {};
			}
			\foreach \n in {5,...,7} {
				\draw (\n) node [dot,label=below:\n] {};
			}
			\draw (2) -- (1);
			\draw (3) -- (4);
			\draw (4) -- (5);
			\draw (4) -- (6);
			\draw (4) -- (7);
		\end{scope}
		
		\begin{scope}[shift={(0.5,0)},scale=0.75]
			\coordinate (2) at (0,0) {};
			\coordinate (3) at (-0.75,-1) {};
			\coordinate (4) at (0.75,-1) {};
			\coordinate (1) at (0,-2) {};
			\coordinate (5) at (1.5,-2) {};
			\coordinate (6) at (0.75,-3) {};
			\coordinate (7) at (0,-4) {};
			\foreach \n in {1,...,7} {
				\draw (\n) node [dot,label=left:\n] {};
			}
			\draw (2) -- (3);
			\draw (3) -- (1);
			\draw (2) -- (4);
			\draw (4) -- (5);
			\draw (5) -- (6);
			\draw (6) -- (7);
		\end{scope}
		
		\begin{scope}[shift={(3.5,-3)},scale=0.5]
			\draw [help lines] (0,0) grid (7,7);
			\draw[line width=1pt] 
			(0,0) -- (0,1) -- (0,2) -- (1,2) -- (1,3) --
			(2,3) -- (3,3) -- (3,4) -- (4,4) -- (4,5) -- 
			(4,6) -- (4,7) -- (5,7) -- (6,7) -- (7,7);
			\node () at (0.5,0.5) {2};
			\node () at (0.5,1.5) {3};
			\node () at (1.5,2.5) {1};
			\node () at (3.5,3.5) {4};
			\node () at (4.5,4.5) {5};
			\node () at (4.5,5.5) {6};
			\node () at (4.5,6.5) {7};
		\end{scope}
	\end{tikzpicture}
	\caption{A labeled rooted forest with its corresponding binary rooted tree and Dyck path}
\end{figure}
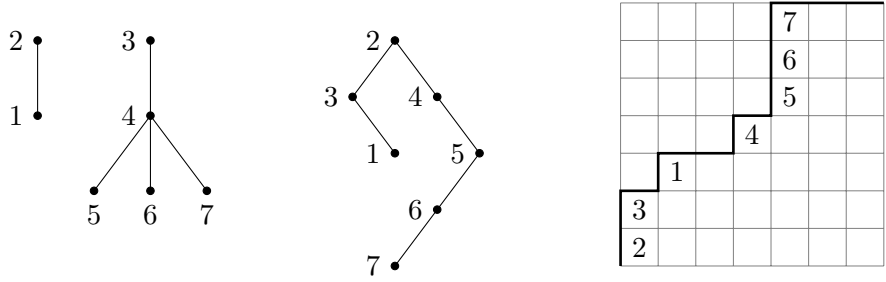

\begin{defin}
  \label{forestbijection}
  Let $f$ be a labeled rooted forest.  Define $\phi : \RF_n \rightarrow \PF_n$ as follows:
  \begin{enumerate}
    \item First, we construct a (labeled) binary rooted tree $b$ on $n$ vertices using the information from $f$. For 
    each vertex $v$ in $f$, we will define a corresponding vertex $v^\prime$ in $b$ with the same label.
    \begin{enumerate}
      \item If $v$ is the (unique) vertex in $f$ such that $h(v)= 0$ and $\eta(v) = 1$, then define 
      $v^\prime$ to be the root of $b$. 
      \item If $v$ has a sibling $w$ such that $\eta(w) = \eta(v)+1$, then $v^\prime$ will have a left-edge connecting
      $v^\prime$ to $w^\prime$ such that $w^\prime$ is a child of $v^\prime$. 
      \item If $v$ has any children, let $x$ be the unique child such that such that $\eta(x) = 1$. Then, $v^\prime$ 
      has a right-edge connecting $v^\prime$ to $x^\prime$ such that $x^\prime$ is a child of $v^\prime$.
    \end{enumerate}
  \item We now use $b$ to construct a labeled Dyck path.  Start at the root of $b$ and proceed counterclockwise around the outside of the tree, visiting each vertex twice. Each time we encounter a vertex that we have not already seen, we construct an up-step in our lattice path with the same label.  If we first encounter a vertex immediately before we encounter a left-edge, we construct  a right-step after encountering the vertex for the second time.  Otherwise, we construct a right-step immediately after recording the up-step corresponding to the vertex. 
  \item We can now use the bijection described in Proposition \ref{dyck} to find the corresponding parking function
  $\phi(f)$.
  \end{enumerate}
\end{defin}
We now define a way of partitioning these rooted forests that behaves nicely with respect to $\phi$.  

\begin{defin}
  Let $f_1$, $f_2$ be labeled rooted forests and define an equivalence relation by $f_1 \sim f_2$ if and only if there exists a graph isomorphism $\psi:f_1 \rightarrow f_2$ such that  for all vertices $v$ in $f_1$, $h(v) = h(\psi(v))$ and $\eta(v) = \eta(\psi(v))$. 
\end{defin}

\begin{prop}
  $\phi$ induces a one-to-one correspondence between equivalence classes of $\RF_n$ and equivalence classes of $\PF_n$ under permutation.
\end{prop}
\begin{proof}
  Let $f_1$ and $f_2$ be labeled rooted forests such that $f_1 \sim f_2$ and let $\psi : f_1 \rightarrow f_2$ be the associated graph isomorphism. Since $h$ and $\eta$ are preserved by the isomorphism, the corresponding binary trees $b_1$ and $b_2$ are identical up to relabeling. Hence, the corresponding labeled Dyck paths have the same shape and thus correspond to parking functions
  which are rearragements of one another.
\end{proof}

We can now realize our decomposition of parking functions in the setting of labeled rooted forests.  
\begin{defin}
  Let $\RF_{n,k}$ be the set of labeled rooted forests which correspond to the parking functions in $\PF_{n,k}$ via
  the bijection above, and let $\RF_{n,1}$
 be those which correspond to parking functions identified with $\sym_n$.
\end{defin}
First, we notice that the recurrence relation between $\PF_{n,k}$ and $\PF_{n+1,k}$ has an analogue in this setting.

\begin{defin}
  Each labeled rooted forest has a unique path from the root to a leaf such that every vertex $v_i$ on this path satisfies $\eta(v_i) = 1$. We denote this path as $P_f$ and its leaf as $v_f$.
\end{defin}

\begin{prop}
  \label{forestrecursion}
  The rooted forest equivalence classes in $\RF_{n+1,k}$ can be obtained from the equivalence classes in $\RF_{n,k}$ by adding a single edge and leaf to $v_f$ for each $f \in \RF_{n,k}$.  
\end{prop}

\begin{proof}
  Fix an equivalence class in $\RF_{n+1,k}$ and select a representative $f$ so that the vertex $v_f$ is labeled
  $n+1$. Notice that in the binary tree $b$ corresponding to $f$, $P_f$ becomes the maximal path of right-edges 
  from the root. Removing $v_f$ produces a labeled rooted forest $f^\prime \in \RF_n$ with the property
  that the associated binary tree $b^\prime$ differs only from $b$ in that its path of right-edges from the root has 
  one fewer
  vertex. In other words, the Dyck paths associated to $f$ and $f^\prime$ are the identical except for an additional
  ``up-right" pair at the end of the former, so by the bijection in Definition \ref{forestbijection}, we can see that
  $f^\prime \in \RF_{n,k}$.  Rearranging the labels for $f$ would change nothing except the labels on $f^\prime$,
  so any equivalence class in $\RF_{n+1,k}$ can be obtained from one in $\RF_{n,k}$ by the procedure above.
\end{proof}

We can use similar reasoning to determine which equivalence classes of rooted forests are contained in each $\RF_{n,k}$.

\begin{figure}
	\begin{tikzpicture}
		\begin{scope}[shift={(0,0)}]
			\node () at (-1.5,-0.5) {\Large $\RF_{3,1}$:};
			\node[draw,plate,minimum size=2.2cm] () at (0.75,-0.5) {};
			\coordinate (1) at (0.75,0.25) {};
			\coordinate (2) at (0.75,-0.5) {};
			\coordinate (3) at (0.75,-1.25) {};
			\draw[red] (1) -- (2);
			\draw[red] (2) -- (3);
			\foreach \n in {1,2,3} {
				\draw (\n) node [dot, label=left:\n] {};
			}
		\end{scope}
		\begin{scope}[shift={(0.375,-2.5)}]
			\node () at (-1.875,-0.5) {\Large $\RF_{3,2}$:};
			\node[draw,plate,minimum size=2.2cm] () at (0.375,-0.5) {};
			\coordinate (1) at (0,0) {};
			\coordinate (2) at (0.75,0) {};
			\coordinate (3) at (0,-1) {};
			\draw[red] (1) -- (3);
			\foreach \n in {1,2} {
				\draw (\n) node [dot, label=above:\n] {};
			}
			\draw (3) node [dot,label=left:3] {};
		\end{scope}
		\begin{scope}[shift={(0,-5)}]
			\node () at (-1.5,-0.5) {\Large $\RF_{3,3}$:};
			\node[draw,plate,minimum size=2.2cm] () at (0.75,-0.5) {};
			\coordinate (1) at (0,0) {};
			\coordinate (2) at (0.75,0) {};
			\coordinate (3) at (1.5,0) {};
			\foreach \n in {1,2,3} {
				\draw (\n) node [dot, label=above:\n] {};
			}
		\end{scope}
		\begin{scope}[shift={(3,-5)}]
			\node[draw,plate,minimum size=2.2cm] () at (0.375,-0.5) {};
			\coordinate (1) at (0,0) {};
			\coordinate (2) at (0.75,0) {};
			\coordinate (3) at (0.75,-1) {};
			\draw (2) -- (3);
			\foreach \n in {1,2} {
				\draw (\n) node [dot, label=above:\n] {};
			}
			\draw (3) node [dot,label=right:3] {};
		\end{scope}
		\begin{scope}[shift={(5,-5)}]
			\node[draw,plate,minimum size=2.2cm] () at (1,-0.5) {};
			\coordinate (1) at (1,0) {};
			\coordinate (2) at (0.5,-1) {};
			\coordinate (3) at (1.5,-1) {};
			\draw (1) -- (2);
			\draw (1) -- (3);
			\draw (1) node [dot,label=above:1] {};
			\draw (2) node [dot,label=left:2] {};
			\draw (3) node [dot,label=right:3] {};
		\end{scope}
	\end{tikzpicture}
	\caption{The equivalence classes of $\RF_3$ split according to the decomposition}
\end{figure}
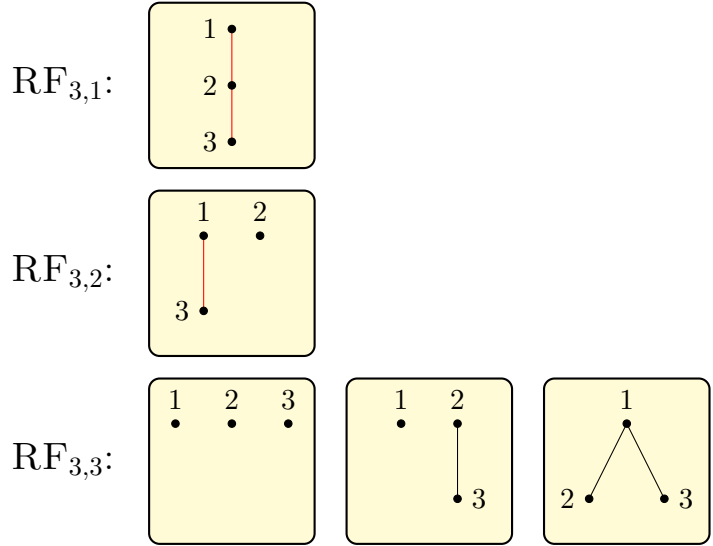

\begin{thm}
  For each $f \in \RF_n$, consider the longest sub-path of $P_f$ containing $v_f$ and not containing a root vertex such that none of the vertices on the path have siblings. Then $f \in \RF_{n,k}$ if and only if the number of vertices
  on this sub-path is $n-k$. 
\end{thm}

\begin{proof}
  By Proposition \ref{forestrecursion}, it suffices to show this for $\RF_{n,n}$. To this end, suppose that $f\in \RF_n$
  and that $v_f$ has at least one sibling. Then, under the bijection given in Definition \ref{forestbijection}, the 
  vertex associated to $v_f$ in the corresponding binary tree has a child via a left-edge, but no right-edge. 
  It follows that the labeled Dyck path for this tree ends with (at least) two right-steps, and thus the parking function
  $\phi(f)$ does not contain $n$, hence $f\in \RF_{n,n}$. By reversing our steps through the bijection we can
  see that $v_f$ has a sibling if and only if $f\in \RF_{n,n}$ and our claim is proven.
\end{proof}

It seems that these are simply a few of many fruitful ways of studying our decomposition of $\PF_n$. We expect that there are several other settings or generalizations for these ideas which could give an interesting perspective for the structure of parking functions.

\section*{Acknowledgements} We would like to thank the National Science Foundation for their support via grant no. DMS-1358884 and the University of California, Santa Barbara for their support through the Research Experience for Undergraduates program. We would also like to thank Maribel Bueno Cachadina, Padraic Bartlett, and Jon McCammond for their guidance and helpful feedback. The second author would like to thank Andy Wilson for his talk at the 2015 AMS Spring Sectional Meeting \cite{andytalk}, which helped inspire the development of this REU project. More details for his use of parking functions with undesired spaces can be found in \cite{andythesis}.

\bibliographystyle{alpha}
\bibliography{paper_final}

\end{document}